\documentclass[11pt, reqno]{amsart}

\usepackage{amscd,amsmath}
\usepackage{mathrsfs}
\usepackage{amsfonts}
\usepackage{amssymb}
\usepackage{enumerate}
\usepackage{setspace}

\usepackage{color}

\usepackage[hyperindex, colorlinks=true]{hyperref}

\usepackage[margin=1in]{geometry}

\newcommand{\eqn}{\begin{eqnarray}}
\newcommand{\een}{\end{eqnarray}}

\newtheorem{theorem}{Theorem}[section]

\newtheorem{lemma}{Lemma}[section]

\newtheorem{definition}[theorem]{Definition}

\numberwithin{equation}{section}

\begin{document}

\title[Transport equation with nonlocal velocity]{On the local and global existence of solutions to 1D transport equations with nonlocal velocity}

\author{Hantaek Bae}
\address{Department of Mathematical Sciences, Ulsan National Institute of Science and Technology (UNIST), Republic of Korea}
\email{hantaek@unist.ac.kr}

\author{Rafael Granero-Belinch\'{o}n}
\address{Departamento  de  Matem\'aticas,  Estad\'istica  y  Computaci\'on,  Universidad  de Cantabria,  Avda.  Los  Castros  s/n,  Santander,  Spain.}
\email{rafael.granero@unican.es}

\author{Omar Lazar}
\address{Departamento de An\'alisis Matem\'atico \& IMUS Universidad de Sevilla C/ Tarfia s/n, Campus Reina Mercedes, 41012 Sevilla, Spain}
\email{omarlazar@us.es}

\date{\today}



\begin{abstract}
We consider the 1D transport equation with nonlocal velocity field:
\begin{equation*}\label{intro eq}
\begin{split}
&\theta_t+u\theta_x+\nu \Lambda^{\gamma}\theta=0, \\
& u=\mathcal{N}(\theta),
\end{split}
\end{equation*}
where $\mathcal{N}$ is a nonlocal operator. In this paper, we show the existence of solutions of this model locally and globally in time for various types of nonlocal operators.
\end{abstract}

\maketitle

\vspace{-2em}

\section{Introduction} \label{sec:1}
In this paper, we study transport equations with nonlocal velocity. One of the most well-known equation is  the two dimensional Euler equation in vorticity form, 
\[
\omega_{t}+u\cdot \nabla \omega=0,
\]
where the velocity $u$ is recovered from the vorticity $\omega$ through
\[
u=\nabla^\perp(-\Delta)^{-1}\omega \quad \text{or equivalently} \quad  \widehat{u}(\xi)=\frac{i\xi^{\perp}}{|\xi|^{2}}\widehat{\omega}(\xi).
\]
Other nonlocal and quadratically nonlinear equations, such as the surface quasi-geostrophic equation, the incompressible porous medium equation, Stokes equations, magneto-geostrophic equation in multi-dimensions, have  been studied intensively as one can see in  \cite{Bae 2, Bae, Baker, Carrillo, CC, CC2, Chae, De Gregorio, Kiselev, Lazar2, LiRodrigo, LiRodrigo2, Morlet} and references therein.

We here consider the 1D transport equations with nonlocal velocity field of the form 
\begin{subequations}\label{model equation}
\begin{align}
&\theta_t+u\theta_x+\nu \Lambda^{\gamma}\theta=0,\\
& u=\mathcal{N}(\theta),
\end{align}
\end{subequations}
where $\mathcal{N}$ is typically expressed by a Fourier multiplier. The study of (\ref{model equation}) is mainly motivated by \cite{CCF} where C\'ordoba, C\'ordoba, and Fontelos proposed the following 1D model 
\begin{subequations}\label{CCF}
\begin{align}
&\theta_{t}+u\theta_{x}=0, \\
& u=-\mathcal{H}\theta, \quad (\text{$\mathcal{H}$ being the Hilbert transform})
\end{align}
\end{subequations}
for the 2D surface quasi-geostrophic equation and  proved the finite time blow-up of smooth solutions. In this paper, we deal with (\ref{CCF}) and its variations with the following objectives.
\begin{enumerate}[]
\item (1) The existence of weak solution with \emph{rough initial data}. The existence of global-in-time solutions is possible even if strong solutions blow up in finite time, as in the case of the Burgers' equation.  
\item (2) The existence of strong solution when the velocity $u$ is more singular than $\theta$. We intend to see  the competitive relationship between nonlinear terms and viscous terms. 
\end{enumerate}

More specifically, the topics covered in this paper can be summarized as follows.

\vspace{1ex}

\noindent
\textbullet \ {\bf The model 1: $\mathcal{N}=-\mathcal{H}$ and $\nu=0$.} We first show the existence of local-in-time solution in a critical space under the scaling $\theta_{0}(x)\mapsto \theta_{0}(\lambda x)$.  We then introduce the notion of a weak super-solution and obtain a global-in-time weak super-solution with $\theta_{0}\in L^{1}\cap L^{\infty}$ and $\theta_{0}\ge 0$.

\vspace{1ex}

\noindent
\textbullet \ {\bf The model 2: $\mathcal{N}=-\mathcal{H}(\partial_{xx} )^{-\alpha}$, $\alpha>0$, $\nu=1$, and $\gamma>0$.} This is a regularized version of (\ref{CCF}) which is also closely related to many equations as mentioned in \cite{Bae 3}. In this case, we show the existence of weak solutions globally in time under weaker conditions on $\alpha$ and $\gamma$ compared to \cite{Bae 3}.

\vspace{1ex}

\noindent
\textbullet \ {\bf The model 3: $\mathcal{N}=-\mathcal{H}(\partial_{xx} )^{\beta}$, $\beta>0$, $\nu=1$, and $\gamma>0$.} Since $\beta>0$, the velocity field is more singular than the previous two models.  In this case, we show the existence of strong solutions locally in time in two cases: (1) $0<\beta\leq \frac{\gamma}{4}$ when $0<\gamma<2$ and (2) $0<\beta<1$ when $\gamma=2$. We also show the existence of strong solutions for $0<\beta<\frac{1}{2}$ and $\gamma=2$ with rough initial data. We finally show the existence of strong solutions globally in time with $0<\beta<\frac{1}{4}$ and $\gamma=2$.

\vspace{1ex}

We will give detailed statements and proofs of our results in Section 3--5.

\section{Preliminaries}
All constants will be denoted by $C$ that is a generic constant. In a series of inequalities, the value of $C$ can vary with each inequality. We use following notation: for a Banach space $X$,
\[
C_{T}X=C([0,T]:X), \quad L^{p}_{T}X=L^{p}(0,T:X).
\]

The Hilbert transform is defined as 
\[
\mathcal{H}f(x)=\text{p.v.} \int_{\mathbb{R}} \frac{f(y)}{x-y}dy.
\]
We will use the BMO space (see e.g. \cite{Bahouri} for the definition) and its dual which is the Hardy space $\mathcal{H}^1$ which consists of those $f$ such that $f$ and $\mathcal{H}f$ are integrable. We will use the following formula
\[
2 \mathcal{H}(f\mathcal{H}f)=(\mathcal{H}f)^2 - f^2
\]
which implies that $g=f\mathcal{H}f \in \mathcal{H}^1$ and for any $f\in L^{2}$,
\begin{equation} \label{hardy}
\Vert g \Vert_{\mathcal{H}^1}  \leq \Vert f \Vert^{2}_{L^{2}}.
\end{equation}

The differential operator $\Lambda^{\gamma}=(\sqrt{-\Delta})^{\gamma}$ is defined by the action of the following kernels \cite{Cordoba}:
\eqn \label{lambda gamma}
\Lambda^{\gamma} f(x)=c_{\gamma}\text{p.v.} \int_{\mathbb{R}} \frac{f(x)-f(y)}{|x-y|^{1+\gamma}}dy,
\een
where $c_{\gamma}>0$ is a normalized constant. Alternatively, we can define $\Lambda^{\gamma}=(\sqrt{-\Delta})^{\gamma}$ as a Fourier multiplier: $\widehat{\Lambda^{\gamma} f}(\xi)=|\xi|^{\gamma}\widehat{f}(\xi)$. When $\gamma=1$,  $\Lambda f(x)=\mathcal{H}f_{x}(x)$.

\vspace{1ex}

We finally introduce Simon's compactness.

\begin{lemma} \cite{Simon} \label{lem:2.2}
Let $X_{0}$, $X_{1}$, and $X_{2}$ be Banach spaces such that $X_{0}$ is compactly embedded in $X_{1}$ and $X_{1}$ is a subset of $X_{2}$. Then, for $1\leq p<\infty$, the set $\left\{v\in L^{p}_{T}X_{0}: \ \frac{\partial v}{\partial t}\in L^{1}_{T}X_{2}\right\}$ is compactly embedded in $L^{p}_{T}X_{1}$.
\end{lemma}

\section{The model 1}

We now study (\ref{model equation}) with $\mathcal{N}=-\mathcal{H}$ and $\nu=0$ which is nothing but (\ref{CCF}):
\begin{subequations} \label{eq:1.1}
\begin{align}
& \theta_{t} -\left(\mathcal{H}\theta\right)\theta_{x} =0,  \label{eq:1.1 a}\\
& \theta(0,x)=\theta_{0}(x). \label{eq:1.1 b}
\end{align}
\end{subequations}

\subsection{Local well-posedness}
The local well-posedness of (\ref{eq:1.1}) is established in $H^{2}$ (\cite{Bae}) and  $H^{\frac{3}{2}-\gamma}$ with the viscous term $\Lambda^{\gamma}\theta$ (\cite{HDong}). To improve these results, we first notice that (\ref{eq:1.1}) has the following scaling invariant property: if $\theta(t,x)$ is a solution of (\ref{eq:1.1}), then so is $\theta_{\lambda}(t,x)=\theta(\lambda t, \lambda x)$. So, we take initial data in a space whose norm is closely invariant under the scaling: $\theta_{0}(x)\mapsto \theta_{\lambda 0}(x)=\theta_{0}(\lambda x)$. In this paper, we take the space $\dot{B}^{\frac{3}{2}}_{2,1}$ because there is a constant $C$ such that 
\[
C^{-1} \left\|\theta_{\lambda 0}\right\|_{\dot{B}^{\frac{3}{2}}_{2,1}} \leq \|\theta_{0}\|_{\dot{B}^{\frac{3}{2}}_{2,1}} \leq C \left\|\theta_{\lambda 0}\right\|_{\dot{B}^{\frac{3}{2}}_{2,1}}. 
\]
The mathematical tools needed to prove the local well-posedness of (\ref{eq:1.1}), such as the Littlewood-Paley decomposition and Besov spaces, are provided in the appendix. We also need the following commutator estimate \cite[Lemma 2.100, Remark 2.101]{Bahouri}.

\begin{lemma}[Commutator estimate]\label{commutator lemma 1}
For $f, g \in \mathcal{S}$
\[
\left\|[f,\Delta_{j}]g_{x}\right\|_{L^{2}}\leq Cc_{j}2^{-\frac{3}{2}j} \left\|f_{x}\right\|_{\dot{B}^{\frac{1}{2}}_{2,1}} \left\|g\right\|_{\dot{B}^{\frac{3}{2}}_{2,1}}, \quad \sum^{\infty}_{j=-\infty}c_{j}\leq 1.
\]
\end{lemma}

The first result in this paper the following theorem.

\begin{theorem}\label{Besov theorem}
For any $\theta_{0}\in \dot{B}^{\frac{3}{2}}_{2,1}$, there exists $T=T(\|\theta_0\|)$ such that a unique solution of (\ref{eq:1.1}) exists in  $C_{T}\dot{B}^{\frac{3}{2}}_{2,1}$. 
\end{theorem} 

\begin{proof}
We only provide a priori estimates of $\theta$ in the space stated in Theorem \ref{Besov theorem}.  The other parts, including the approximation procedure, are rather standard. 

We apply $\Delta_{j}$ to (\ref{eq:1.1}), multiply by  $\Delta_{j}\theta$, and integrate the resulting equation over $\mathbb{R}$ to get 
\begin{equation} \label{eq:3.2}
\begin{split}
\frac{1}{2}\frac{d}{dt} \left\|\Delta_{j} \theta \right\|^2_{L ^{2}} &=\int_{\mathbb{R}} \Delta_{j}\left((\mathcal{H}\theta)\theta_{x} \right) \Delta_{j}\theta dx\\
&=\int_{\mathbb{R}} \left((\mathcal{H}\theta)\Delta_{j}\theta_{x} \right) \Delta_{j}\theta dx+ \int_{\mathbb{R}} \left[\Delta_{j}, \mathcal{H}\theta\right]\Delta_{j}\theta_{x} \Delta_{j}\theta dx\\
&=-\frac{1}{2}\int_{\mathbb{R}} (\mathcal{H}\theta)_{x}\left|\Delta_{j}\theta\right|^{2}  dx+ \int_{\mathbb{R}} \left[\Delta_{j}, \mathcal{H}\theta\right]\Delta_{j}\theta_{x} \Delta_{j}\theta dx.
\end{split}
\end{equation}
By the Bernstein inequality, we have
\eqn \label{low frequency}
\left\|\mathcal{H}\theta_{x}\right\|_{L^{\infty}}\leq C\|\theta\|_{\dot{B}^{\frac{3}{2}}_{2,1}}.
\een
We then apply Lemma \ref{commutator lemma 1} to the second term in the right-hand side of (\ref{eq:3.2}) to obtain  
\begin{equation}\label{eq:3.3}
\begin{split}
\int_{\mathbb{R}} \left[\Delta_{j}, \mathcal{H}\theta\right]\Delta_{j}\theta_{x} \Delta_{j}\theta dx\leq Cc_{j}2^{-\frac{3}{2}j}\|\theta\|^{2}_{\dot{B}^{\frac{3}{2}}_{2,1}} \left\|\Delta_{j} \theta \right\|_{L ^{2}}.
\end{split}
\end{equation}
By (\ref{eq:3.2}), (\ref{low frequency}), and (\ref{eq:3.3}), we have 
\[
\frac{d}{dt}\|\theta\|^{2}_{\dot{B}^{\frac{3}{2}}_{2,1}}\leq C\|\theta\|^{3}_{\dot{B}^{\frac{3}{2}}_{2,1}},
\]
from which we deduce
\[
\|\theta(t)\|_{\dot{B}^{\frac{3}{2}}_{2,1}}\leq \frac{\|\theta_{0}\|_{\dot{B}^{\frac{3}{2}}_{2,1}}}{1-Ct\|\theta_{0}\|_{\dot{B}^{\frac{3}{2}}_{2,1}}} \leq 2 \|\theta_{0}\|_{\dot{B}^{\frac{3}{2}}_{2,1}} \quad \text{for all $t\leq T=\frac{1}{2C \|\theta_{0}\|_{\dot{B}^{\frac{3}{2}}_{2,1}}}$}.
\]
This completes the proof.
\end{proof}

\subsection{Global weak super-solution}
We next consider (\ref{eq:1.1}) with rough initial data. More precisely, we assume that $\theta_{0}$ satisfies the following conditions
\eqn \label{initial condition}
\theta_{0}\ge 0, \quad \theta_{0}\in L^{1}\cap L^{\infty}.
\een
Since $\theta$ satisfies the transport equation, we have
\eqn
\theta(t,x)\ge 0, \quad \theta \in L^{\infty}(\mathbb{R}) \quad \text{for all time}.
\een

If we follow the usual weak formulation of (\ref{eq:1.1}), for all $\phi\in C^{\infty}_{c}([0,\infty)\times \mathbb{R})$
\eqn \label{weak solution to 1D sqg}
\int^{T}_{0}\int_{\mathbb{R}} \left[\theta \psi_{t} - \left(\mathcal{H}\theta\right)\theta \psi_{x}+\left(\Lambda \theta\right) \theta \psi \right]  dxdt =\int_{\mathbb{R}}\theta_{0}(x)\psi(x,0)dx.
\een
For $\theta_{0}\ge 0$, there is gain of a half derivative from the structure of the nonlinearity, that is
\eqn \label{half energy dd}
\left\|\theta(t)\right\|_{L^{1}} +\int^{t}_{0}\left\|\Lambda^{\frac{1}{2}}\theta(s)\right\|^{2}_{L^{2}}ds=\left\|\theta_{0}\right\|_{L^{1}}.
\een
So, we can rewrite the left-hand side of (\ref{weak solution to 1D sqg}) as 
\[
\int^{T}_{0}\int_{\mathbb{R}} \left[\theta \psi_{t} - \left(\mathcal{H}\theta\right)\theta \psi_{x}+\Lambda^{\frac{1}{2}}\theta \left[\Lambda^{\frac{1}{2}},\psi\right]\theta +\left|\Lambda^{\frac{1}{2}}\theta\right|^{2}\psi\right]  dxdt =\int_{\mathbb{R}}\theta_{0}(x)\psi(x,0)dx.
\]
However,  the $\dot{H}^{\frac{1}{2}}$ regularity derived from (\ref{half energy dd}) is not enough to pass to the limit in 
\[
\int^{T}_{0}\int_{\mathbb{R}} \left|\Lambda^{\frac{1}{2}}\theta^{\epsilon}\right|^{2}\psi dxdt
\]
from the $\epsilon$-regularized equations  described below. So, we introduce a new notion of solution. Let
\[
\mathcal{A}_{T}=L^{\infty}_{T}\left(L^{1}\cap L^{\infty}\right)\cap L^{2}_{T}H^{\frac{1}{2}}.
\]

\begin{definition}
We say $\theta$ is a weak super-solution of (\ref{eq:1.1}) on the time interval $[0,T]$ if $\theta(t,x)\ge 0$ for all $t\in [0,T]$, $\theta\in \mathcal{A}_{T}$, and for each nonnegative $\psi\in C^{\infty}_{c}([0,T]\times\mathbb{R})$,
\eqn \label{supersolution to 1d sqg}
\int^{T}_{0}\int_{\mathbb{R}} \left[\theta \psi_{t} - \left(\mathcal{H}\theta\right)\theta \psi_{x}+\Lambda^{\frac{1}{2}}\theta \left[\Lambda^{\frac{1}{2}},\psi\right]\theta +\left|\Lambda^{\frac{1}{2}}\theta\right|^{2}\psi\right]  dxdt \ge \int_{\mathbb{R}}\theta_{0}(x)\psi(x,0)dx. 
\een
\end{definition}

To prove Theorem \ref{main theorem}, we need to estimate a commutator term involving $\Lambda^{\frac{1}{2}}$: 
\[
\left[\Lambda^{1/2},\psi\right](f-g)\in L^{6}
\]
which is proved in \cite{Bae 3}.

\begin{lemma} \label{commutator lemma 0}
For $f\in L^{\frac{3}{2}}$, $g\in L^{\frac{3}{2}}$  and $\psi \in W^{1,\infty}$, we have
\[
\left\| \left[\Lambda^{\frac{1}{2}},\psi \right]f- \left[\Lambda^{\frac{1}{2}},\psi \right]g\right\|_{L^{6}}\leq C\|\psi\|_{W^{1,\infty}} \left\|f-g\right\|_{L^{\frac{3}{2}}}.
\]
\end{lemma}

The second result in our paper is the following theorem.

\begin{theorem}\label{main theorem}
For any  $\theta_{0}$ satisfying (\ref{initial condition}), there exists a weak super-solution of (\ref{eq:1.1}) in $\mathcal{A}_{T}$.
\end{theorem}

\begin{proof}
We first regularize initial data as $\theta^{\epsilon}_{0}=\rho_{\epsilon}\ast \theta_{0}$ where $\rho_{\epsilon}$ is a standard mollifier that preserve the positivity of the regularized initial data. We then regularize the equation by introducing the Laplacian term with a coefficient $\epsilon>0$, namely
\eqn \label{regularized equation}
\theta^{\epsilon}_t -\mathcal{H}\theta^{\epsilon}\theta^{\epsilon}_x =\epsilon \theta^{\epsilon}_{xx}.
\een
For the proof of the existence of a global-in-time smooth solution we refer to \cite{Lazar}. Moreover, $\theta^{\epsilon}$ satisfies that $\theta^{\epsilon}\ge 0 $ and 
\[
\left\| \theta^{\epsilon}(t) \right\|_{L^{1}}+\left\| \theta^{\epsilon}(t) \right\|_{L^{\infty}}+ \int^{t}_{0}\left\|\Lambda^{\frac{1}{2}}\theta^{\epsilon}(s)\right\|^{2}_{L^{2}}ds \leq  \left\|\theta_{0}\right\|_{L^{1}} + \left\|\theta_{0}\right\|_{L^{\infty}}.
\]
Therefore, $(\theta_{\epsilon})$ is bounded in $\mathcal{A}_{T}$ uniformly in $\epsilon>0$. 

From this, we have uniform bounds 
\[
\mathcal{H}\theta^{\epsilon}\in L^{4}_{T}L^{2}, \quad \theta^{\epsilon}\in L^{2}_{T}L^{2}, \quad \left(\left(\mathcal{H}\theta^{\epsilon}\right)\theta^{\epsilon}\right)_x\in L^{\frac{4}{3}}_{T}H^{-2}, \quad \epsilon \theta^{\epsilon}_{xx}\in L^{2}_{T}H^{-2}.
\]
Moreover,  for any $\phi \in H^{2}$,
\[
\int_{\mathbb{R}}\left|\theta^{\epsilon} \Lambda\theta^{\epsilon} \phi \right|  dx \leq \left\|\Lambda^{\frac{1}{2}}\theta^{\epsilon}\right\|^{2}_{L^{2}} \left\|\phi\right\|_{L^{\infty}} + \left\|\Lambda^{\frac{1}{2}}\theta^{\epsilon}\right\|_{L^{2}} \left\|\theta^{\epsilon}\right\|_{L^{\infty}} \left\|\Lambda^{\frac{1}{2}}\phi\right\|_{L^{\infty}} 
\]
which implies that 
\[
\theta^{\epsilon} \Lambda\theta^{\epsilon} \in L^{1}_{T}H^{-2}.
\]

Combining all together, we obtain 
\[
\theta^{\epsilon}_{t}=\mathcal{H}\theta^{\epsilon}\theta^{\epsilon}_x +\epsilon \theta^{\epsilon}_{xx}= \left(\mathcal{H}\theta^{\epsilon}\theta^{\epsilon}\right)_x -\theta^{\epsilon} \Lambda\theta^{\epsilon}+\epsilon \theta^{\epsilon}_{xx}\in L^{1}_{T}H^{-2}.
\]

To pass to the limit into the weak super-solution formulation, we extract a subsequence of $\left(\theta^{\epsilon}\right)$, using the same index $\epsilon$ for simplicity, and a function $\theta \in \mathcal{A}_T$ such that 
\begin{equation}\label{convergence of approximate solutions}
\begin{split}
& \theta^{\epsilon} \stackrel{\star}{\rightharpoonup} \theta \quad \text{in} \quad L^{\infty}_{T}\left(L^{p}\cap H^{\frac{1}{2}}\right) \quad \text{for all $p\in (1,\infty)$},\\
& \theta^{\epsilon} \rightharpoonup \theta \quad \text{in} \quad L^{2}_{T}H^{\frac{1}{2}},\\
& \theta^{\epsilon} \rightarrow  \theta \quad \text{in $L^{2}_{T}L^{p}$ for all $1<p<\infty$} ,
\end{split}
\end{equation}
where we use Lemma \ref{lem:2.2} for the strong convergence with  
\[
X_0=L^2_{T}H^{\frac{1}{2}},\quad X_{1}=L^2_{T}L^p,\quad X_2=L^1_{T}H^{-2}.
\]

We now multiply (\ref{regularized equation}) by a test function $\psi\in \mathcal{C}^{\infty}_{c}\left([0,T)\times\mathbb{R}\right)$ and integrate over $\mathbb{R}$. Then,
\begin{equation} \label{weak form with commutator 111}
\begin{split}
&\int^{T}_{0}\int \Big[\theta^{\epsilon} \psi_{t} -  \underbrace{\left(\mathcal{H}\theta^{\epsilon}\right)\theta^{\epsilon}\psi_{x}}_{\text{I}}+ \epsilon \theta^{\epsilon}\psi_{xx}\Big]  dxdt - \int \theta^{\epsilon}_{0}(x)\psi(0,x)dx\\
& = -\int^{T}_{0}\int \underbrace{\Lambda^{\frac{1}{2}}\theta^{\epsilon} \left[\Lambda^{\frac{1}{2}},\psi \right]\theta^{\epsilon}}_{\text{II}}  dxdt -\int^{T}_{0}\int \underbrace{\left|\Lambda^{\frac{1}{2}}\theta^{\epsilon}\right|^{2}\psi}_{\text{III}}  dxdt.
\end{split}
\end{equation}

We note that we are able to rearrange terms in the usual weak formulation into (\ref{weak form with commutator 111}) since $\theta^{\epsilon}$ is smooth.  By the strong convergence in (\ref{convergence of approximate solutions}), we can pass to the limit to $\text{I}$. Moreover, since
\[
\left[\Lambda^{\frac{1}{2}},\psi \right]\theta^{\epsilon} \rightarrow \left[\Lambda^{\frac{1}{2}},\psi \right]\theta
\]
strongly in $L^{2}_{T}L^{6}$ by Lemma \ref{commutator lemma 0} and the strong convergence in (\ref{convergence of approximate solutions}), we can pass to the limit to $\text{II}$. Lastly, by Fatou's lemma, 
\[
\lim_{\epsilon\rightarrow 0}\int^{T}_{0}\int \left|\Lambda^{\frac{1}{2}}\theta^{\epsilon}\right|^{2}\psi dxdt \ge \int^{T}_{0}\int \left|\Lambda^{\frac{1}{2}}\theta\right|^{2}\psi dxdt.
\]
Combining all the limits together, we obtain that 
\begin{equation} \label{weak form with commutator 2}
\begin{split}
\int^{T}_{0}\int_{\mathbb{R}} \left[\theta \psi_{t} - \left(\mathcal{H}\theta\right)\theta \psi_{x}+\Lambda^{\frac{1}{2}}\theta \left[\Lambda^{\frac{1}{2}},\psi\right]\theta +\left|\Lambda^{\frac{1}{2}}\theta\right|^{2}\psi\right]  dxdt \ge \int_{\mathbb{R}}\theta_{0}(x)\psi(x,0)dx. 
\end{split}
\end{equation}
This completes the proof. 
\end{proof}

\section{The model 2}

We now consider the following equation:
\begin{subequations} \label{new model}
\begin{align}
& \theta_{t} -\left(\mathcal{H}(\partial_{xx})^{-\alpha}\theta\right)\theta_{x} +\Lambda^{\gamma}\theta=0, \\
& \theta(0,x)=\theta_{0}(x),
\end{align}
\end{subequations}
where $\alpha,\gamma>0$. In this case, we focus on the existence of weak solutions under some conditions  of $(\alpha,\gamma)$. As before, we assume that $\theta_{0}$ satisfies the following conditions
\eqn \label{initial condition 2}
\theta_{0}\ge 0, \quad \theta_{0}\in L^{1}\cap L^{\infty}.
\een

Let
\[
\mathcal{B}_{T}=L^{\infty}_{T}\left(L^{1}\cap L^{\infty}\right)\cap L^{2}_{T}H^{\frac{\gamma}{2}}.
\]

\begin{definition} 
We say $\theta$ is a weak solution of (\ref{new model}) on the time interval $[0,T]$ if $\theta(t,x)\ge 0$ for all $t\in [0,T]$, $\theta\in \mathcal{B}_{T}$, and for each $\psi\in C^{\infty}_{c}([0,T]\times\mathbb{R})$,
\[
\int^{T}_{0}\int_{\mathbb{R}} \left[\theta \psi_{t} - \left(\mathcal{H}(\partial_{xx})^{-\alpha}\theta\right)\theta \psi_{x}-\Lambda^{1-\frac{\gamma}{2}} (\partial_{xx})^{-\alpha}\theta \Lambda^{\frac{\gamma}{2}} (\theta\psi) -\theta \Lambda^{\gamma}\psi\right]  dxdt = \int_{\mathbb{R}}\theta_{0}(x)\psi(x,0)dx. 
\]
\end{definition}

The third result in the paper is the following.

\begin{theorem}\label{new model weak}
Suppose that two positive numbers $\alpha$ and $\gamma$ satisfy
\eqn \label{range of parameters}
0<\gamma<1, \quad \alpha\ge \frac{1}{2}-\frac{\gamma}{2}.
\een
Then, for any  $\theta_{0}$ satisfying (\ref{initial condition 2}), there exists a weak solution of (\ref{new model}) in $\mathcal{B}_{T}$ for all $T>0$.
\end{theorem}

\begin{proof}
As in the proof of Theorem \ref{main theorem}, we regularize $\theta_{0}$ and the equation as
\eqn \label{regularized equation new model}
\theta^{\epsilon}_{0}=\rho_{\epsilon}\ast \theta_{0}, \quad \theta^{\epsilon}_t -\left(\mathcal{H}(\partial_{xx})^{-\alpha}\theta^{\epsilon}\right)\theta^{\epsilon}_{x} +\Lambda^{\gamma}\theta^{\epsilon} =\epsilon \theta^{\epsilon}_{xx}.
\een
Then, the corresponding $\theta^{\epsilon}$ satisfies 
\eqn \label{sign and max}
\theta^{\epsilon}(t,x)\ge 0, \quad \left\|\theta^{\epsilon}(t)\right\|_{L^{\infty}}\leq  \|\theta_{0}\|_{L^{\infty}} \quad \text{for all time }
\een
and 
\eqn \label{L1 bound new} 
\left\|\theta^{\epsilon}(t)\right\|_{L^{1}} +\int^{t}_{0}\left\|\Lambda^{\frac{1}{2}}(\partial_{xx})^{-\frac{\alpha}{2}}\theta^{\epsilon}(s)\right\|^{2}_{L^{2}}ds\leq \left\|\theta_{0}\right\|_{L^{1}}.
\een
We next multiply (\ref{regularized equation new model}) by $\theta^{\epsilon}$ and integrate over $\mathbb{R}$. Then,  
\begin{equation*}
\begin{split}
&\frac{1}{2}\frac{d}{dt}\left\|\theta^{\epsilon}(t)\right\|^{2}_{L^{2}}+\left\|\Lambda^{\frac{\gamma}{2}}\theta^{\epsilon}(t)\right\|^{2}_{L^{2}} +\epsilon \left\|\theta^{\epsilon}_{x}\right\|^{2}_{L^{2}}=-\frac{1}{2}\int_{\mathbb{R}}\left\{\Lambda (\partial_{xx})^{-\alpha}\theta^{\epsilon}(t)\right\} (\theta^{\epsilon}(t))^{2}dx\\
&=-\frac{1}{2}\int_{\mathbb{R}}\left\{\Lambda^{1-\frac{\gamma}{2}} (\partial_{xx})^{-\alpha}\theta^{\epsilon}(t)\right\} \Lambda^{\frac{\gamma}{2}}(\theta^{\epsilon}(t))^{2}dx\\
& \leq C\left\|\Lambda^{1-\frac{\gamma}{2}} (\partial_{xx})^{-\alpha}\theta^{\epsilon}(t)\right\|_{L^{2}} \left\|\Lambda^{\frac{\gamma}{2}}\theta^{\epsilon}(t)\right\|_{L^{2}}\|\theta^{\epsilon}(t)\|_{L^{\infty}} \\
& \leq \frac{1}{2} \left\|\Lambda^{\frac{\gamma}{2}}\theta^{\epsilon}(t)\right\|^{2}_{L^{2}} + C\left\|\Lambda^{1-\frac{\gamma}{2}} (\partial_{xx})^{-\alpha}\theta^{\epsilon}(t)\right\|^{2}_{L^{2}}\|\theta^{\epsilon}(t)\|^{2}_{L^{\infty}}.
\end{split}
\end{equation*}
By (\ref{range of parameters}), (\ref{sign and max}) and (\ref{L1 bound new}), we obtain 
\eqn \label{new L2 bound}
\|\theta^{\epsilon}(t)\|^{2}_{L^{2}}+\int^{t}_{0}\left\|\Lambda^{\frac{\gamma}{2}}\theta^{\epsilon}(s)\right\|^{2}_{L^{2}}ds +\epsilon \int^{t}_{0}\left\|\theta^{\epsilon}_{x}(s)\right\|^{2}_{L^{2}}ds \leq C\|\theta_{0}\|^{2}_{L^{1}} \|\theta_{0}\|^{2}_{L^{\infty}}.
\een
Therefore, $(\theta_{\epsilon})$ is bounded in $\mathcal{B}_{T}$ uniformly in $\epsilon>0$. 

From this, we have uniform bounds 
\[
\left\{\left(\mathcal{H}(\partial_{xx})^{-\alpha}\theta\right)\theta\right\}_{x} \in L^{2}_{T}L^{2}, \quad \Lambda^{\gamma}\theta^{\epsilon}+\epsilon \theta^{\epsilon}_{xx}\in L^{2}_{T}H^{-2}.
\]
Moreover, the condition (\ref{range of parameters}) implies that 
\[
\left(\Lambda(\partial_{xx})^{-\alpha}\theta^{\epsilon}\right)\theta^{\epsilon}  \in L^{1}_{T}H^{-1}.
\]

Combining all together, we also derive that 
\[
\theta^{\epsilon}_{t}\in L^{1}_{T}H^{-2}.
\]

We now multiply (\ref{regularized equation new model}) by a test function $\psi\in \mathcal{C}^{\infty}_{c}\left([0,T)\times\mathbb{R}\right)$ and integrate over $\mathbb{R}$. Then,
\begin{equation} \label{weak form with commutator}
\begin{split}
&\int^{T}_{0}\int \Big[\theta^{\epsilon} \psi_{t} -  \underbrace{\left(\mathcal{H}(\partial_{xx})^{-\alpha}\theta^{\epsilon}\right)\theta^{\epsilon}\psi_{x}}_{\text{I}}+ \Lambda^{\gamma}\theta^{\epsilon}+\epsilon \theta^{\epsilon}\psi_{xx}\Big]  dxdt - \int \theta^{\epsilon}_{0}(x)\psi(0,x)dx\\
& = \int^{T}_{0}\int \underbrace{\Lambda^{1-\frac{\gamma}{2}}\mathcal{H}(\partial_{xx})^{-\alpha}\theta^{\epsilon} \Lambda^{\frac{\gamma}{2}}(\theta^{\epsilon}\psi)}_{\text{II}}  dxdt.
\end{split}
\end{equation}

To pass the limit to this formulation, we extract a subsequence of $\left(\theta^{\epsilon}\right)$, using the same index $\epsilon$ for simplicity, and a function $\theta \in \mathcal{B}_T$ such that 
\begin{equation}\label{convergence of approximate solutions 2}
\begin{split}
& \theta^{\epsilon} \stackrel{\star}{\rightharpoonup} \theta \quad \text{in} \quad L^{\infty}_{T}\left(L^{p}\cap H^{\frac{1}{2}}\right) \quad \text{for all $p\in (1,\infty)$},\\
& \theta^{\epsilon} \rightharpoonup \theta \quad \text{in} \quad L^{2}_{T}H^{\frac{\gamma}{2}},\\
& \theta^{\epsilon} \rightarrow  \theta \quad \text{in $L^{2}_{T}H^{1-\frac{\gamma}{2}-2\alpha}\cap L^{2}_{T}L^{p}$ for all $1<p<\infty$} ,
\end{split}
\end{equation}
where we use Lemma \ref{lem:2.2} for the strong convergence with the condition (\ref{range of parameters}) and 
\[
X_0=L^2_{T}H^{\frac{\gamma}{2}},\quad X_{1}=L^{2}_{T}H^{1-\frac{\gamma}{2}-2\alpha} \cap L^2_{T}L^p,\quad X_2=L^1_{T}H^{-2}.
\]
By the strong convergence in (\ref{convergence of approximate solutions 2}), we can pass to the limit to $\text{I}$ and $\text{II}$ in (\ref{weak form with commutator}). Therefore, we obtain  
\[\int^{T}_{0}\int_{\mathbb{R}} \left[\theta \psi_{t} - \left(\mathcal{H}(\partial_{xx})^{-\alpha}\theta\right)\theta \psi_{x}-\Lambda^{1-\frac{\gamma}{2}} (\partial_{xx})^{-\alpha}\theta \Lambda^{\frac{\gamma}{2}} (\theta\psi) -\theta \Lambda^{\gamma}\psi\right]  dxdt = \int_{\mathbb{R}}\theta_{0}(x)\psi(x,0)dx. 
\]
This completes the proof of Theorem \ref{new model weak}.
\end{proof}

\noindent
{\bf Remark.} Theorem \ref{new model weak} improves Theorem 1.4 in \cite{Bae 3}, where $(\alpha, \gamma)$ is assumed to satisfy $\alpha\ge \frac{1}{2}-\frac{\gamma}{4}$. The main idea of taking weaker  regularization in (\ref{new model}) is that the Hilbert transform in front of $(1-\partial_{xx})^{-\alpha}$ gives (\ref{L1 bound new}) which makes to obtain (\ref{new L2 bound}). We choose $\alpha> \frac{1}{2}-\frac{\gamma}{2}$ instead of $ \alpha\ge \frac{1}{2}-\frac{\gamma}{2}$ to apply  compactness argument when we pass to the limit to $\epsilon$-regularized equations. 

\section{The model 3}

In this section, we consider the following equation 
\begin{subequations} \label{singular model d}
\begin{align}
& \theta_{t} -\left(\mathcal{H}(\partial_{xx})^{\beta}\theta\right)\theta_{x} +\Lambda^{\gamma}\theta=0, \\
& \theta(0,x)=\theta_{0}(x)
\end{align}
\end{subequations}
where $\beta,\gamma>0$. Depending on the range of $\beta$ and $\gamma$, we will have four different results.

\subsection{Local well-posedness}
We begin with the local well-posedness result.

\begin{theorem}  \label{LW and blow-up criterion theorem}
Let $0<\gamma<2$ and $0<\beta\leq \frac{\gamma}{4}$. For $\theta_0 \in H^2 (\Bbb R)$ there exists $T=T(\|\theta_0\|_{H^2})$ such that a unique solution of (\ref{singular model d}) exists in $C\left([0, T); H^2 (\Bbb R)\right)$. Moreover, we have the following blow-up criterion:
\eqn \label{blow-up criterion 1}
\lim\sup_{t\nearrow T^*} \|\theta(t)\|_{H^2}=\infty \quad \text{if and only if} \quad \int_0 ^{T^*} \Big(\left\|u_x(s)\right\|_{L^{\infty}}+\left\|\theta_x(s)\right\|_{L^{\infty}}\Big)ds=\infty.
\een
\end{theorem}

\begin{proof}
Let $u=-\mathcal{H}(\partial_{xx})^{\beta}\theta$. Operating $\partial_x^l$ on (\ref{singular model d}), taking its $L^2$ inner product with $\partial_x^l \theta$, and summing over $l=0, 1,2$, 
\begin{equation} \label{pr1}
\begin{split}
&\frac12 \frac{d}{dt} \|\theta(t)\|_{H^2}^2 + \left\|\Lambda^{\frac{\gamma}{2}}\theta\right\|_{H^2}^2= -\sum_{l=0} ^2 \int \partial_x^l (u \theta_x) \partial_x^l \theta dx \\
 &=-\sum_{l=0} ^2 \int \left( \partial_x^l (u \theta_x)- u  \partial_x^l \theta_x \right) \partial_x^l \theta dx-\sum_{l=0} ^2 \int u  \partial_x^l \theta_x  \partial_x^l \theta    dx = \text{I}_1+ \text{I}_2.
\end{split}
\end{equation}
Using the commutator estimate in \cite{Kato}
\[
\sum_{|l|\leq 2} \left\|D^l (fg)-fD^l g\right\|_{L^2} \leq   C\left(\|\nabla f\|_{L^\infty} \left\|D g\right\|_{L^2} +\left\|D^2 f\right\|_{L^2}   \|g\|_{L^\infty} \right),
\]
we have
\begin{equation}\label{bound of I1}
\begin{split}
 \text{I}_1 \leq \sum_{l=0} ^2 \left\|\partial_x^l (u \theta_x)- u  \partial_x^l   \theta_x\right\|_{L^2} \|\theta\|_{H^2} &\leq C\left(\|u_x\|_{L^\infty} \|\theta\|_{H^2} +   \|u\|_{H^2}\|\theta_x\|_{L^\infty} \right)\|\theta\|_{H^2} \\
   &\leq C_{\kappa}   \left(\|u_x\|_{L^\infty}+\|\theta_x\|^{2}_{L^\infty}\right)\|\theta\|_{H^2}^2 +  \kappa\|u\|^{2}_{H^2}.
\end{split}
\end{equation}
And by integration by parts,
\eqn \label{bound of I2}
\text{I}_2 = -\frac{1}{2}\sum_{l=0} ^2\int u \partial_x \left|\partial_x^l \theta\right|^2 dx= \frac{1}{2}\sum_{l=0} ^2\int u_x \left|\partial_x^l \theta \right|^2 dx \leq  C \|u_x\|_{L^\infty} \|\theta\|_{H^2}^2.
\een
Since $\beta\leq \frac{\gamma}{4}$, for a sufficiently small $\kappa>0$
\[
\kappa\|u\|^{2}_{H^2}\leq \frac{1}{2}\left\|\Lambda^{\frac{\gamma}{2}}\theta\right\|_{H^2}^2.
\]
By (\ref{bound of I1}) and (\ref{bound of I2}), we obtain 
\eqn \label{Hm bound of w}
\frac{d}{dt} \|\theta\|_{H^2}^2 +\left\|\Lambda^{\frac{\gamma}{2}}\theta\right\|_{H^2}^2 \leq C \left(\|u_x\|_{L^\infty}+\|\theta_x\|^{2}_{L^\infty}\right)\|\theta\|_{H^2}^2 \leq C\|\theta\|_{H^2}^3+ C\|\theta\|_{H^2}^4, \quad \beta\leq \frac{\gamma}{4}
\een
from which we deduce that there is $T=T(\|\theta_{0}\|_{H^{2}})$ such that 
\[
\|\theta(t)\|_{H^2}\leq  2 \|\theta_0\|_{H^2} \quad \text{for all} \ t<T.
\]
(\ref{Hm bound of w}) also implies (\ref{blow-up criterion 1}). 
 
To show the uniqueness, let $\theta_{1}$ and $\theta_{2}$ be two solutions of (\ref{singular model d}), and let $\theta=\theta_{1}-\theta_{2}$ and $u=u_{1}-u_{2}$. Then, $(\theta,u)$ satisfies the following equations
\[
\theta_t+u_{1}\theta_x-u\theta_{2x}=-\Lambda^{\gamma}\theta, \quad u=-\mathcal{H}(\partial_{xx} )^{\beta} \theta, \quad \theta(0,x)=0.
\]
By taking the $L^{2}$ product of the equation with $\theta$,  
\[
\frac{d}{dt}\|\theta\|^{2}_{L^{2}}+2\left\|\Lambda^{\frac{\gamma}{2}}\theta\right\|^{2}_{L^{2}} \leq C \left(\left\|u_{1x}\right\|_{L^{\infty}} +\left\|\theta_{2x}\right\|_{L^{\infty}}\right)\|\theta\|^{2}_{L^{2}}\leq C \left(\left\|\Lambda^{\frac{\gamma}{2}}\theta_{1}\right\|_{H^{2}} +\left\|\theta_{2}\right\|_{H^{2}}\right)\|\theta\|^{2}_{L^{2}}.
\]
So, $\theta=0$ in $L^{2}$ and thus a solution  is unique. This completes the proof of Theorem \ref{LW and blow-up criterion theorem}.
\end{proof}

Theorem \ref{LW and blow-up criterion theorem} provides a local existence result for $\beta \nearrow \frac{1}{2}$ as $\gamma\nearrow 2$. But, we can increase the range of $\beta$ when we deal with (\ref{singular model d}) directly with $\gamma=2$ because we can do the integration by parts.

\begin{theorem}  \label{LW and blow-up criterion theorem 2}
Let $\gamma=2$ and  $0<\beta<1$. For $\theta_0 \in H^2 (\Bbb R)$ there exists $T=T(\|\theta_0\|_{H^2})$ such that a unique solution of (\ref{singular model d}) exists in $C\left([0, T); H^2 (\Bbb R)\right)$.
\end{theorem}

\begin{proof}
We begin the $L^{2}$ bound:
\[
\frac{1}{2}\frac{d}{dt}\left\|\theta\right\|^{2}_{L^{2}}+ \left\|\theta_{x}\right\|^{2}_{L^{2}}\leq \|\theta\|_{L^{\infty}} \left\|\mathcal{H}(\partial_{xx})^{\beta}\theta\right\|_{L^{2}} \left\|\theta_{x}\right\|_{L^{2}}\leq C\|\theta\|^{3}_{H^{2}}.
\]
We next estimate $\theta_{xx}$.  Indeed, after several integration parts, we have 
\[
\frac{1}{2}\frac{d}{dt}\left\|\theta\right\|^{2}_{\dot H^{2}} + \Vert \theta \Vert^{2}_{\dot H^{3}}= -\int \left\{\mathcal{H}(\partial_{xx})^{\beta}\theta_{x}\right\}\theta_{x}\theta_{xxx}dx+\frac{1}{2}\int \left\{\mathcal{H}(\partial_{xx})^{\beta}\theta_{x}\right\}\theta_{xx}\theta_{xx}dx =\text{I}_1+\text{I}_2.
\]
When $0<\beta<1$,
\begin{equation*}
\begin{split}
\left|\text{I}_{1}\right|& \leq \left\|\theta_{x}\right\|_{L^{\infty}}\left\|\mathcal{H}(\partial_{xx} )^{\beta} \theta_{x}\right\|_{L^{2}}\left\|\theta_{xxx}\right\|_{L^{2}} =\left\|\theta_{x}\right\|_{L^{\infty}}\left\|\Lambda^{2\beta+1}\theta \right\|_{L^{2}}\left\|\theta_{xxx}\right\|_{L^{2}}\\
& \leq C\left\|\theta\right\|_{H^{2}} \left\|\theta_{x}\right\|^{1-\beta}_{L^{2}}\left\|\theta_{xxx}\right\|^{1+\beta}_{L^{2}} \leq C\left\|\theta\right\|^{4}_{H^{2}} + C\left\|\theta\right\|^{\frac{4-2\beta}{1-\beta}}_{H^{2}}+\frac{1}{4}\left\|\theta_{xxx}\right\|^{2}_{L^{2}}.
\end{split}
\end{equation*}
And
\[
\left|\text{I}_{2}\right| \leq \left\|\mathcal{H}(\partial_{xx})^{\beta}\theta_{x}\right\|_{L^{2}} \left\|\theta_{xx}\right\|^{2}_{L^{4}} \leq C\left\|\mathcal{H}(\partial_{xx})^{\beta}\theta_{x}\right\|_{L^{2}} \left\|\theta_{xx}\right\|^{\frac{3}{2}}_{L^{2}}\left\|\theta_{xxx}\right\|^{\frac{1}{2}}_{L^{2}} \leq C\left\|\theta\right\|^{4}_{H^{2}} +\frac{1}{4}\left\|\theta_{xxx}\right\|^{2}_{L^{2}}.
\]

Therefore, we obtain 
\eqn
\frac{d}{dt}\left\|\theta\right\|^{2}_{H^{2}}+ \left\|\theta_{x}\right\|^{2}_{H^{2}}\leq C\left\|\theta\right\|^{4}_{H^{2}} + C\left\|\theta\right\|^{\frac{4-2\beta}{1-\beta}}_{H^{2}}. 
\een
This implies that there exists $T=T(\|\theta_0\|_{H^2})$ such that there exists a unique solution of (\ref{singular model d}) in $C\left([0, T); H^2 (\Bbb R)\right)$.
\end{proof}

We may lower the regularity of the initial data to prove a local existence result of a weak solution by considering initial data in $\dot H^{\frac{1}{2}}$. The main tools to achieve this will be the use of the Hardy-BMO duality together with interpolation arguments. However, in order to simplify the computation, we consider an equivalent equation by changing the sign of the nonlinearity:
\begin{subequations} \label{singular model d}
\begin{align}
& \theta_{t} + \left(\mathcal{H}(-\partial_{xx})^{\beta}\theta\right)\theta_{x} +\Lambda^{\gamma}\theta=0, \\
& \theta(0,x)=\theta_{0}(x)
\end{align}
\end{subequations}
This can be obtained from (\ref{singular model d}) via $\theta \mapsto -\theta$. For this equation, we do $\dot H^{\frac{1}{2}}$ estimates and prove that there exists a local existence of a unique solution in that special case.

\begin{theorem} \label{GW}
Let $\gamma=2$ and $0<\beta<\frac{1}{2}$. For any  $\theta_0 \in \dot H^{\frac{1}{2}} (\Bbb R)$, there exists $T=T(\Vert \theta_{0} \Vert_{\dot H^{\frac{1}{2}}})$ such that there exists a unique local-in-time solution in $C([0, T); \dot H^{\frac{1}{2}} (\Bbb R)) \cap L^2\left([0, T); H^{\frac{3}{2}} (\Bbb R)\right)$.
\end{theorem}

\begin{proof}
By recalling that $\Lambda^{2\beta}=(-\partial_{xx})^{\beta}$ we get
\begin{equation*}
\begin{split}
\frac{1}{2} \frac{d}{dt} \Vert \theta \Vert^{2}_{\dot H^{\frac{1}{2}}} +\left\Vert \Lambda^{\frac{1+\gamma}{2}} \theta \right\Vert^{2}_{L^{2}} &=- \int \Lambda^{\frac{1}{2}} \theta \Lambda^{\frac{1}{2}}\left\{\left(\mathcal{H}(-\partial_{xx})^{\beta}\theta\right)\theta_{x} \right\} dx  \\
&=- \int \theta_{x}\Lambda \theta \ \mathcal{H}(-\partial_{xx})^{\beta}\theta  dx=- \int \theta_{x} \mathcal{H} \theta_{x}  \mathcal{H}(-\partial_{xx})^{\beta}\theta dx.
\end{split}
\end{equation*}
We now use the $\mathcal{H}^1$-BMO duality to estimate the right hand side of the last equality. By using the estimate \eqref{hardy} and $\dot H^{\frac{1}{2}} \hookrightarrow BMO$, we obtain 
\[
\Vert \theta_{x} \mathcal{H} \theta_{x} \Vert_{\mathcal{H}^1} \leq \Vert \theta \Vert^{2}_{\dot H^{1}}, \quad \left\|\mathcal{H}(-\partial_{xx})^{\beta}\theta\right\|_{L^{2}}\leq C\Vert \theta \Vert_{\dot H^{2\beta +\frac{1}{2}}}
\]
and thus we have 
\[
\frac{1}{2} \frac{d}{dt} \Vert \theta \Vert^{2}_{\dot H^{\frac{1}{2}}} +\left\Vert \Lambda^{\frac{1+\gamma}{2}} \theta \right\Vert^{2}_{L^{2}} \leq  C\Vert  \theta \Vert^{2}_{\dot H^{1}}    \Vert \theta \Vert_{\dot H^{2\beta +\frac{1}{2}}}.
\]
By fixing $\gamma=2$ and by using the interpolation inequalities
\[
\Vert \theta \Vert^{2}_{\dot H^1} \leq \Vert \theta \Vert_{\dot H^{\frac{3}{2}}}\Vert \theta \Vert_{\dot H^{\frac{1}{2}}}, \quad \Vert \theta \Vert_{\dot H^{2\beta +\frac{1}{2}}} \leq \Vert \theta \Vert^{2\beta}_{\dot H^{\frac{3}{2}}} \Vert \theta \Vert^{1-2\beta}_{\dot H^{\frac{1}{2}}},
\]
where we use $\frac{1}{2}\leq 2\beta+\frac{1}{2} \leq \frac{3}{2}$ for $\beta \in \left(0,\frac{1}{2}\right)$ to get the second inequality. Hence, we obtain  
\begin{equation*}
\begin{split}
\frac{1}{2} \frac{d}{dt} \Vert \theta \Vert^{2}_{\dot H^{\frac{1}{2}}} +\left\Vert \Lambda^{\frac{3}{2}} \theta \right\Vert^{2}_{L^{2}} &\leq  \Vert  \theta \Vert^{2}_{\dot H^{1}}    \Vert \theta \Vert_{\dot H^{2\beta + \frac{1}{2}}} \\
& \leq  \Vert \theta \Vert^{1+2\beta}_{\dot H^{\frac{3}{2}}} \Vert \theta \Vert^{2-2\beta}_{\dot H^{\frac{1}{2}}} \leq \frac{1}{2}\Vert \theta \Vert^{2}_{\dot H^{\frac{3}{2}}} + 2\Vert \theta \Vert^{4\frac{1-\beta}{1-2\beta}}_{\dot H^{\frac{1}{2}}},
\end{split}
\end{equation*}
where we use the condition $\beta \in \left(0,\frac{1}{2}\right)$ again to derive the inequality. This implies local existence of a unique solution up to some time $T=T(\Vert \theta_{0} \Vert_{\dot H^{\frac{1}{2}}})$.
\end{proof}

\subsection{Global well-posedness}

We finally deal with (\ref{singular model d}) with $\gamma=2$. 

\begin{theorem} \label{GW}
Let $\gamma=2$ and $\beta<\frac{1}{4}$. For any  $\theta_0 \in H^2 (\Bbb R)$, there exists a unique global-in-time solution in $C\left([0, \infty); H^2 (\Bbb R)\right)$.
\end{theorem}

\begin{proof}
By Theorem \ref{LW and blow-up criterion theorem}, we only need to control the quantities in (\ref{blow-up criterion 1}). Let  $u=-\mathcal{H}(\partial_{xx})^{\beta}\theta$. We first note that (\ref{singular model d}) satisfies the maximum principle and so 
\[
\left\|\theta(t)\right\|_{L^{\infty}}\leq \left\|\theta_{0}\right\|_{L^{\infty}}\leq C \|\theta_{0}\|_{H^{2}}.
\]

We take the $L^2$ inner product of (\ref{singular model d}) with $\theta$. Then, 
\eqn
\frac{1}{2}\frac{d}{dt}\|\theta\|^{2}_{L^{2}} +\|\theta_{x}\|^{2}_{L^{2}}=-\int u\theta_{x}\theta dx\leq \|\theta_{0}\|_{L^{\infty}} \|u\|_{L^{2}}\|\theta_{x}\|_{L^{2}}.
\een
Since 
\[
 \|u\|_{L^{2}} \leq C\|\theta\|^{1-2\beta}_{L^{2}} \|\theta_{x}\|^{2\beta}_{L^{2}} \quad \text{for $\beta< \frac{1}{2}$},
\]
we have
\eqn \label{theta L2 bound}
\|\theta(t)\|^{2}_{L^{2}} +\int^{t}_{0}\|\theta_{x}(s)\|^{2}_{L^{2}}ds\leq C\left(t, \|\theta_{0}\|_{H^{2}}\right). 
\een

\noindent We next take $\partial_x $ to (\ref{singular model d}), take its $L^2$ inner product with $\theta_x$, and integrate by parts to obtain
\begin{equation*}
\begin{split}
\frac12 \frac{d}{dt}\|\theta_x\|_{L^2} ^2 + \left\|\theta_{xx}\right\|_{L^2}^2 =\int u \theta_x \theta_{xx}dx  \leq 2\|u\|^{2}_{L^{\infty}}\|\theta_{x}\|^{2}_{L^{2}}+\frac{1}{2}\left\|\theta_{xx}\right\|_{L^2}^2.
\end{split}
\end{equation*}
Since 
\[
\|u\|^{2}_{L^{\infty}} \leq C\|\theta\|^{2}_{L^{2}}+C\|\theta_{x}\|^{2}_{L^{2}} \quad \text{when $\beta<\frac{1}{4}$},
\]
we obtain  
\eqn \label{theta H1 bound}
\|\theta_{x}(t)\|^{2}_{L^{2}} +\int^{t}_{0}\|\theta_{xx}(s)\|^{2}_{L^{2}}ds\leq C\left(t, \|\theta_{0}\|_{L^{1}}, \|\theta_{0}\|_{H^{2}}\right) \quad \text{when $\beta<\frac{1}{4}$.}
\een
By (\ref{theta L2 bound}) and (\ref{theta H1 bound}), we finally obtain 
\[
\int^{t}_{0}\left(\|\theta_{x}(s)\|_{L^{\infty}}+ \|u_{x}(s)\|_{L^{\infty}}\right)ds \leq C\int^{t}_{0}\left(\|\theta_{x}(s)\|_{L^{2}}+ \|\theta_{xx}(s)\|_{L^{2}}\right)ds \leq C\left(t, \|\theta_{0}\|_{L^{1}}, \|\theta_{0}\|_{H^{2}}\right)
\]
and so we complete the proof of Theorem \ref{GW}.
\end{proof}

\section{Appendix}
This appendix is briefly written based on \cite{Bahouri}. We first provide notation and definitions in the Littlewood-Paley theory.  Let $\mathcal{C}$ be the ring of center 0, of small radius $\frac{3}{4}$ and great radius $\frac{8}{3}$. We take smooth radial functions $(\chi, \phi)$ with values in $[0,1]$  that are supported on the ball $B_{\frac{4}{3}}(0)$ and $\mathcal{C}$, respectively, and satisfy
\begin{equation}
\begin{split}
& \chi(\xi)+ \sum^{\infty}_{j=0}\phi\left(2^{-j}\xi\right)=1 \ \  \forall \ \xi \in \mathbb{R}^{d},\\
&  \sum^{\infty}_{j=-\infty}\phi\left(2^{-j}\xi\right)=1 \ \  \forall \ \xi \in \mathbb{R}^{d}\setminus\{0\},\\
& \left|j-j^{'}\right|\ge 2  \ \Longrightarrow \ \text{supp}\ \phi\left(2^{-j}\cdot\right)\bigcap \text{supp}\ \phi\left(2^{-j^{'}}\cdot\right)=\emptyset,\\
& j\ge 1  \ \Longrightarrow \ \text{supp}\ \chi \bigcap \text{supp}\ \phi\left(2^{-j}\cdot\right)=\emptyset.
\end{split}
\end{equation}
From now on, we use the notation
\[
\phi_{j}(\xi)=\phi\left(2^{-j}\xi\right).
\]
We define dyadic blocks and lower frequency cut-off functions. 
\begin{equation}
\begin{split}
& h=\mathcal{F}^{-1}\phi, \quad \widetilde{h}=\mathcal{F}^{-1}\chi,\\
& \Delta_{j}f=\phi_{j}\left(D\right)f=2^{jd} \int_{\mathbb{R}^{d}} h\left(2^{j}y\right)f(x-y)dy,\\
& S_{j}f=\chi\left(2^{-j}D\right)f=2^{jd} \int_{\mathbb{R}^{d}} \widetilde{h}\left(2^{j}y\right)f(x-y)dy,\\
& \Delta_{-1}f=\chi\left(D\right)f=\int_{\mathbb{R}^{d}} \widetilde{h}\left(y\right)f(x-y)dy.
\end{split}
\end{equation}
Then, the homogeneous Littlewood-Paley decomposition is given by
\[
f=\sum_{j\in \mathbb{Z}} \Delta_{j}f \ \  \text{in} \ \  \mathcal{S}^{'}_{h},
\]
where $\mathcal{S}^{'}_{h}$ is the space of tempered distributions $u\in \mathcal{S}^{'}$ such that 
\[
\lim_{j\rightarrow -\infty}S_{j}u=0\quad \text{in $\mathcal{S}'$}.
\]

We now define the homogeneous Besov spaces:
\[
\dot{B}^{s}_{p,q} =\left\{f\in \mathcal{S}^{'}_{h}: \ \left\|f\right\|_{\dot{B}^{s}_{p,q}}=\left\|\left\|2^{js}\left\|\Delta_{j}f\right\|_{L^{p}}\right\|_{l^{q}(\mathbb{Z})} \right\|<\infty \right\}.
\]

We also recall Bernstein's inequality in 1D : for $1\leq p\leq q \leq \infty$ and $k\in \mathbb{N}$,
\eqn\label{bernstein}
\sup_{|\alpha|=k} \left\|\partial^{\alpha}\Delta_{j}f \right\|_{L^{p}} \leq C 2^{jk} \left\|\Delta_{j}f \right\|_{L^{p}}, \quad \left\|\Delta_{j}f \right\|_{L^{q}} \leq C 2^{j\left(\frac{1}{p}-\frac{1}{q}\right)} \left\|\Delta_{j}f \right\|_{L^{p}}.
\een

\section*{Acknowledgments}
H.B. was supported by the National Research Foundation of Korea (NRF-2015R1D1A1A01058892). 

R.G.B. was supported by the LABEX MILYON (ANR-10-LABX-0070) of Universit\'e de Lyon, within the program ``Investissements d'Avenir'' (ANR-11-IDEX-0007) operated by the French National Research Agency (ANR), and by the Universidad de Cantabria. 

O.L. was partially supported by the Marie-Curie Grant, acronym: TRANSIC, from the FP7-IEF program and  by the ERC through the Starting Grant project H2020-EU.1.1.-63922.

Both O. L. and R.G.B. were partially supported by the Grant MTM2014-59488-P from the former Ministerio de Econom\'ia y Competitividad (MINECO, Spain).


\end{document}